\documentclass[a4paper, 11pt]{amsart}

\usepackage[mathscr]{eucal}
\usepackage{amsmath}
\usepackage{amssymb}
\usepackage{amsfonts}
\usepackage{amsthm}

\usepackage{cases}

\usepackage[mathscr]{eucal}
\usepackage{amsmath}
\usepackage{amssymb}
\usepackage{amsfonts}
\usepackage{amsthm}
\usepackage[dvips]{graphics, color}
\theoremstyle{plain}
\newtheorem{theorem}{Theorem}[section]
\newtheorem{lemma}{Lemma}[section]
\newtheorem{remark}{Remark}[section]

\renewcommand{\labelenumi}{(\theenumi)}
\numberwithin{equation}{section}

\def\<{\left<} \def\>{\right>}

\def\bea{\begin{eqnarray} }
\def\eea{\end{eqnarray} }
\def\be{\begin{equation} }
\def\ee{\end{equation} }
\def\qed{\ifhmode\unskip\nobreak\fi\ifmmode\ifinner\else\hskip5pt
\fi\fi\hbox{\hskip5 pt \vrule width4 pt height6 pt depth1.5 pt \hskip1pt }}

\begin{document}
\title[]{Ricci curvature of  real hypersurfaces in non-flat complex space forms}
\author[]{Toru Sasahara}
\address{Division of Mathematics, 
Hachinohe Institute of Technology, 
Hachinohe, Aomori 031-8501, Japan}
\email{sasahara@hi-tech.ac.jp}


\begin{abstract} 
We establish an inequality among the Ricci curvature, the squared mean curvature, and the normal curvature for real hypersurfaces in complex space forms.
We classify real hypersurfaces in two-dimensional non-flat complex space forms which admit a
unit vector field satisfying identically the equality case of the inequality. 
\end{abstract}

\keywords{Real hypersurfaces, Ricci curvature, normal curvature.}

\subjclass[2010]{Primary 53C42; Secondary 53B25.}

\maketitle

 \section{Introduction}
 Let  $\tilde M^n(4c)$ denote an $n$-dimensional complex space form
 of constant holomorphic sectional curvature $4c$ $(\ne 0)$, that is, the complex projective $n$-space
 $\mathbb{C}P^n(4c)$ or the complex hyperbolic space $\mathbb{C}H^n(4c)$, according as
 $c>0$ or $c<0$.
 We denote by $J$  the almost complex structure on $\tilde M^n(4c)$.
 Let $M$ be a real hypersurface of  $\tilde M^n(4c)$.
 For a unit normal vector field $N$ of  $M$ in $\tilde M^n(4c)$,
the characteristic vector field on $M$ is defined by $\xi=-JN$.
If $\xi$ is a principal curvature vector at $p\in M$, then $M$ is said to be 
{\it Hopf} at $p$.  If $M$ is Hopf at every point, then $M$ is called a {\it Hopf hypersurface}.

Let $\mathcal{H}$ be the holomorphic distribution defined by $\mathcal{H}_p=\{X\in T_pM\ |\ \<X, \xi\>=0\}$ for $p\in M$.
If $\mathcal{H}$ is integrable and each leaf of its maximal integral manifolds is 
locally congruent to
a totally geodesic complex hypersurface
$\tilde M^{n-1}(4c)$ in $\tilde M^n(4c)$, then $M$ is called a {\it ruled real hypersurface}.

We denote by $\mathcal{D}_p$ the smallest subspace of $T_pM$ that contains $\xi$ and is invariant under the shape operator.
A hypersurface in  $\tilde M^n(4c)$ is said to be 2-{\it Hopf}
 if the distribution defined by $\mathcal{D}_p$ for $p\in M$
  is integrable and of constant rank $2$.

For a unit vector $X\in T_pM$,  we define
  the  {\it normal curvature } in the direction of $X$ by $\kappa_X=\<AX, X\>$.
  We establish in Section 2 an inequality among the Ricci curvature, the squared mean curvature, and 
  the normal curvature (Lemma \ref{lem1}).  By applying this inequality, 
  in case $n=2$, we have
 \be
Ric(\xi)\leq \frac{9}{8}||H||^2+\kappa_{\xi}^2+2c, \label{ricci-xi}
\ee 
  and for any unit vector  $U\in\mathcal{H}_p$,
   \be
Ric(U)\leq \frac{9}{8}||H||^2+\kappa_{U}^2+5c. \label{ricci-U}
\ee
 Here, $Ric (X)$ is the Ricci curvature in the direction $X$ and $H$ the mean curvature vector.
 
 The purpose of this paper is to investigate   real hypersurfaces
 which satisfy ($\ref{ricci-xi}$) or admit a unit vector field 
$U\in\mathcal{H}$
satisfying  (\ref{ricci-U}) identically.
We prove the following.
  \begin{theorem}\label{thm1}
Let $M$ be a real  hypersurface 
   in  $\tilde M^2(4c)$. Then
  the equality sign of $(\ref{ricci-xi})$ holds identically if and only if 
  $M$ is locally congruent to one of the following: 
  \begin{enumerate}
\renewcommand{\labelenumi}{(\roman{enumi})}
\item a geodesic sphere of  radius $\pi/(6\sqrt{c})$ in $\mathbb{C}P^2(4c)$, 
\item a tube of radius $\ln(2+\sqrt{3})/(4\sqrt{|c|})$ over a totally geodesic real hyperbolic space
  $\mathbb{R}H^2$ in $\mathbb{C}H^2(4c)$. 
\end{enumerate}
  
\end{theorem} 

\begin{theorem}\label{thm2}
Let $M$ be a Hopf hypersurface 
   in $\tilde M^2(4c)$.
Then $M$ admits a unit vector field $U\in\mathcal{H}$
 satisfying the equality  in $(\ref{ricci-U})$ identically if and only if $M$ is locally congruent to one of the following:
\begin{enumerate}
\renewcommand{\labelenumi}{(\roman{enumi})}
\item a tube  of radius $\pi/(6\sqrt{c})$ over a totally geodesic real projective space 
$\mathbb{R}P^2$ in $\mathbb{C}P^2(4c)$,
\item a horosphere in $\mathbb{C}H^2(4c)$.
\end{enumerate}
\end{theorem}

In the non-Hopf case, under the assumption  that a vector field $U\in\mathcal{H}$ 
satisfying identically (\ref{ricci-U})
 is a geodesic vector field on $M$, that is, the integral curves of $U$ are geodesics on $M$, we have the following.
\begin{theorem}\label{thm3}
Let $M$ be a real hypersurface 
   in  $\tilde M^2(4c)$ which 
 is non-Hopf at every point.
Then $M$ admits a unit geodesic vector field $U\in\mathcal{H}$  satisfying the equality in 
$(\ref{ricci-U})$ identically if and only if $M$ is locally congruent to
a $2$-Hopf hypersurface in $\mathbb{C}P^2$ or $\mathbb{C}H^2$ 
such that $\<A\xi, \xi\>$ is constant on $M$.
\end{theorem}

By applying Theorem \ref{thm3}, we  prove the following.
\begin{theorem}\label{thm4}
Let $M$ be a real hypersurface with  constant mean curvature
   in  $\tilde M^2(4c)$ which 
 is non-Hopf at every point.
Then $M$ admits a unit vector field $U\in\mathcal{H}$  satisfying the equality in 
$(\ref{ricci-U})$ identically if and only if $M$ is locally congruent to one of the
following{\rm :}
\begin{enumerate}
\renewcommand{\labelenumi}{(\roman{enumi})}
\item a minimal ruled hypersurface in $\mathbb{C}P^2(4c)$ or $\mathbb{C}H^2(4c)$, 
\item a $2$-Hopf hypersurface in $\mathbb{C}P^2(4c)$ such that the shape operator $A$ is represented by a matrix
\be
A=
    \begin{pmatrix}
       -7\sqrt{\frac{c}{8}}&  \sqrt{\frac{27c}{8}}\tan\Bigl(\sqrt{\frac{27c}{8}}s+d\Bigr)& 0\\
       \sqrt{\frac{27c}{8}}\tan\Bigl(\sqrt{\frac{27c}{8}}s+d\Bigr) & \sqrt{\frac{c}{8}}&0 \\
       0 & 0 & -\sqrt{\frac{c}{2}}
   \end{pmatrix}\nonumber
\ee
with respect to an orthonormal frame field $\{\xi, X, JX\}$, where $d$ is some constant and 
 $JX=\partial/\partial s$,
\item one of the equidistant hypersurfaces to a Lohnherr hypersurface in $\mathbb{C}H^2(4c)$.
\end{enumerate}
\end{theorem}

 \begin{remark}
 {\rm There exist infinity many hypersurfaces which satisfy the conditions of Theorem \ref{thm3}.
 In fact,  such  hypersurfaces can be constructed by solutions of the system (\ref{2-hopf}) of ODE's in Section 4 such that $\alpha$ is constant. 
 If $\alpha$ and $\gamma$ in the system are constant, then the corresponding hypersurfaces are
 the ones described in Theorem \ref{thm4}}.
\end{remark}

 \section{Preliminaries}
 \subsection{Basic formulas and fundamental equations}
Let $M$ be a   real hypersurface in  $\tilde M^n(4c)$.
Let us denote by $\nabla$ and
 $\tilde\nabla$ the Levi-Civita connections on $M$ and $\tilde M^n(4c)$, respectively. The
 Gauss and Weingarten formulas are respectively given by
\be
 \begin{split}
 \tilde \nabla_XY&= \nabla_XY+\<AX, Y\>N, \label{gawe}\\
 \tilde\nabla_X N&= -AX
 \end{split}\nonumber
\ee
 for tangent vector fields $X$, $Y$ and a unit normal vector field $N$,
 where $A$ is the shape operator.
The mean curvature vector field $H$ is defined by 
$H=({\rm Tr}(A)/(2n-1))N.$
 The function ${\rm Tr}(A)/(2n-1)$ is called the  {\it mean curvature}.
 If it vanishes identically, then $M$ is called a {\it minimal hypersurface}.
 
For any  vector field  $X$ tangent to $M$,  we denote the tangential component of $JX$ by $\phi X$.
Then it follows from $\tilde\nabla J=0$ and  the Gauss and  Weingarten formulas that
\be
\nabla_{X}\xi=\phi AX. \label{PA}
\ee

 
 We denote by $R$ the Riemannian curvature tensor of $M$. Then,
 the equations of Gauss  and Codazzi are respectively given by
 \begin{align}
 &R(X, Y)Z=c[\<Y, Z\>X-\<X, Z\>Y+\<\phi Y, Z\>\phi X
 -\<\phi X, Z\>\phi Y  \label{ga}\\
 &\hskip60pt -2\<\phi X, Y\>\phi Z]
   +\<AY, Z\>AX-\<AX, Z\>AY,\nonumber\\
  &(\nabla_XA)Y-(\nabla_YA)X=c[\<X, \xi\>\phi Y-\<Y, \xi\>\phi X-2\<\phi X, Y\>\xi].\label{co}
 \end{align}

 \subsection{An inequality concerning the Ricci curvature}
 We recall the following algebraic lemma.
 \begin{lemma}[\cite{deng}]\label{lemma2}
Let $f: \mathbb{R}^{2n-1}\rightarrow\mathbb{R}$ be a function  defined by
\be
f(x_1, x_2, \ldots, x_{2n-1})=x_{2n-1}\sum_{j=1}^{2n-2}x_j-(x_{2n-1})^2. \nonumber
\ee 
Then we have the following inequality{\rm :}
\be
f(x_1, x_2, \ldots, x_{2n-1})\leq \frac{(x_1+x_2+\cdots+x_{2n-1})^2}{8}.\nonumber
\ee
The equality sign holds if and only if $x_1+\cdots+x_{2n-2}=3x_{2n-1}$.
\end{lemma}

 Applying Lemma \ref{lemma2}, we have the following lemma.
\begin{lemma}\label{lem1}
 Let $M$ be a real hypersurface in $\tilde M^n(4c)$.
  Then, for any point $p\in M$ and   any unit vector  $X\in T_pM$, we have   
\be
Ric(X)\leq \frac{(2n-1)^2}{8}||H||^2+\kappa_X^2+c(2n-2+3||\phi X||^2). \label{ricci}
\ee 
The equality sign of $(\ref{ricci})$ holds at $p\in M$ if and only if 
there exists an orthonormal basis $\{e_1, e_2,  \ldots , e_{2n-1}\}$ in $T_pM$ such that $e_{2n-1}=X$ and the shape operator of $M$ in 
$\tilde M^n(4c)$ at $p$ is represented by a matrix
\be
A=
    \begin{pmatrix}
       & & & 0\\
       & \mbox{\smash{\Large\textit{B}}}& &\vdots\\
       & & & 0\\
     0 & \ldots & 0 &  \mu
   \end{pmatrix},\label{A}
\ee
where   $B$ is a symmetric $(2n-2)\times (2n-2)$ submatrix such that
\be\mathrm{Tr}(B)=3\mu. \label{B}
\ee 
 \end{lemma}
 \begin{proof}
 Let $M$ be a real hypersurface in a complex space form $\tilde M^n(4c)$.
 Let $X$ be any unit tangent vector at $p\in M$. 
 We choose an orthonormal basis $\{e_1, \ldots, e_{2n-1}\}$ in $T_pM$
 such 
 that $e_{2n-1}=X$. 
We put $x_j=\<Ae_j, e_j\>$. Then by using 
the equation (\ref{ga}) of Gauss and Lemma \ref{lemma2} we obtain
\be 
\begin{split}
 Ric(X)-c(2n-2+3||\phi X||^2)&=x_{2n-1}\sum_{j=1}^{2n-2}x_j-\sum_{j=1}^{2n-2}\<Ae_j, e_{2n-1}\>^2 \\
 &\leq x_{2n-1}\sum_{j=1}^{2n-2}x_j\\
 &=f(x_1, x_2, \ldots, x_{2n-1})+(x_{2n-1})^2\\
 &\leq \frac{(x_1+x_2+\cdots+x_{2n-1})^2}{8}+(x_{2n-1})^2\\
 &= \frac{(2n-1)^2}{8}||H||^2+\kappa_X^2
\end{split} \label{ineq}
 \ee
This proves inequality (\ref{ricci}). The equality sign of  (\ref{ricci}) holds if and only if 
two inequalities in (\ref{ineq}) become equalities. The application of Lemma \ref{lemma2} implies that the shape operator can be represented by (\ref{A}) with (\ref{B}).
\qed
\end{proof}

By putting $n=2$, $X=\xi$ and $X=U\in\mathcal{H}$ in (\ref{ricci}),  
we can immediately obtain (\ref{ricci-xi}) and (\ref{ricci-U}), respectively.

\section{Proof of Theorems 1.1 and 1.2}

 In order to prove Theorems 1.1 and 1.2, we first state 
 several well-known results concerning  Hopf hypersurfaces.
 We denote by $\delta$ the principal curvature corresponding to $\xi$.
 The following facts are fundamental (see Corollary 2.3 of \cite{ni}).
 \begin{theorem}\label{eigen}
 Let $M$ be a Hopf hypersurface in  $\tilde M^n(4c)$. 
 Then we have
 \begin{itemize}
 \item[(i)] $\delta$ is constant{\rm ;}
 \item[(ii)] If $X$ is a  tangent vector of $M$ orthogonal to $\xi$  such that 
  $AX=\lambda_1X$ and $A\phi X=\lambda_2\phi X$, then
 $2\lambda_1\lambda_2=(\lambda_1+\lambda_2)\delta+2c$ holds.
 \end{itemize} 
 \end{theorem}

By  results of \cite{kim} and  \cite{takagi}, for $n=2$ and $c>0$
we have the following.
\begin{theorem}\label{hopf1} 
Let $M$ be a Hopf hypersurface with constant principal curvatures in $\mathbb{C}P^2(4c)$.
Then $M$ is locally congruent to one of the following: 
\begin{itemize}
\item[$(A_1)$] a geodesic sphere with  radius $r$, where $0<r<\pi/(2\sqrt{c})$, 
\item[$(B)$] a tube of radius $r$ over a totally real totally geodesic $\mathbb{R}P^2$, where $0<r<\pi/(4\sqrt{c})$.
\end{itemize}
\end{theorem}

\begin{theorem}\label{A1cp}
 The Type $(A_1)$ hypersurafces in $\mathbb{C}P^2(4c)$ have two distinct principal curvatures: 
$\delta=2\sqrt{c}\cot(2\sqrt{c}r)$ of multiplicity $1$ and 
$\lambda=\sqrt{c}\cot(\sqrt{c}r)$
of multiplicity $2$.
\end{theorem}

\begin{theorem}\label{Bcp}
  The Type $B$ hypersurfaces in  $\mathbb{C}P^2(4c)$ have
   three distinct principal curvatures: 
$\delta=2\sqrt{c}\tan(2\sqrt{c}r)$,
$\lambda_1=-\sqrt{c}\cot(\sqrt{c}r)$, and 
 $\lambda_2=
  \sqrt{c}\tan(\sqrt{c}r)$. 
\end{theorem}

By  a result of \cite{be}, for $n=2$ and $c<0$ we have the following.
\begin{theorem}\label{hopf2}
Let $M$ be a Hopf hypersurface with constant principal curvatures in $\mathbb{C}H^2(4c)$.
 $M$ is  locally congruent to one of the following: 
\begin{itemize}
\item[$(A_0)$] a horosphere, 
\item[$(A_{1,0})$] a geodesic sphere of radius $r$, where $0<r<\infty$, 
\item[$(A_{1,1})$] a tube of  radius $r$ over a totally geodesic $\mathbb{C}H^1$, where $0<r<\infty$, 
\item[$(B)$] a tube of radius $r$ over a totally real totally geodesic $\mathbb{R}H^2$, where $0<r<\infty$.
\end{itemize}
\end{theorem}

\begin{theorem}\label{A0}
The Type $(A_0)$ hypersurfaces in $\mathbb{C}H^2(4c)$
have two distinct principal curvatures: 
$\delta=2\sqrt{|c|}$ of multiplicity $1$ and
$\lambda=\sqrt{|c|}$ 
of multiplicity $2$.
\end{theorem}

\begin{theorem}\label{A10}
The Type $(A_{1,0})$ hypersurfaces in $\mathbb{C}H^2(4c)$ have two distinct principal curvatures: 
 $\delta=2\sqrt{|c|}\coth(2\sqrt{|c|}r)$  of multiplicity $1$ and 
$\lambda=\sqrt{|c|}\coth(\sqrt{|c|}r)$
of multiplicity $2$.
\end{theorem}

\begin{theorem}\label{A11}
The Type $(A_{1,1})$ hypersurfaces
in  $\mathbb{C}H^2(4c)$ have two distinct principal curvatures: 
 $\delta=2\sqrt{|c|}\coth(2\sqrt{|c|}r)$  of multiplicity $1$ and 
$\lambda=\sqrt{|c|}\tanh(\sqrt{|c|}r)$
of multiplicity $2$. 
\end{theorem}

\begin{theorem}\label{Bch}
The Type $(B)$ hypersurfaces
 in $\mathbb{C}H^2(4c)$ have three  principal curvatures: 
$\delta=2\sqrt{|c|}\tanh(2\sqrt{|c|}r)$,
$\lambda_1=\sqrt{|c|}\coth(\sqrt{|c|}r)$, and 
$\lambda_2=\sqrt{|c|}\tanh(\sqrt{|c|}r)$.
\end{theorem}

By applying the above-mentioned results, we shall prove Theorems 1.1 and 1.2.

\medskip
\noindent{\it Proof of Theorem} 1.1.
Let $M$ be a real hypersurface  in 
$\tilde M^2(4c)$. 
Assume that $\xi$ satisfies the equality in (\ref{ricci-xi}) identically. 
Then,  (\ref{A}) implies that
$M$ is a Hopf hypersurface with $\delta=\mu$. 
It follows from  Theorem \ref{eigen} that
$\mu$ is constant and 
$2\mathrm{det}(B)=\mathrm{Tr}(B)\mu+2c$. 
Combining this and (\ref{B}) shows  that
$M$ is a Hopf hypersurface
with constant principal curvatures, which is one of the hypersurfaces
described in Theorems \ref{hopf1} and \ref{hopf2}.
By applying  Theorems \ref{A1cp}, \ref{Bcp}, \ref{A0}, \ref{A10}, \ref{A11} and \ref{Bch}, we can determine hypersurfaces satisfying 
  (\ref{B}), that is,   
  $3\delta=
2\lambda$ or $\lambda_1+\lambda_2$ as follows:

If $c>0$, then 
$M$ is  a hypersurface of type $A_1$ in $\mathbb{C}P^2(4c)$ satisfying  
\be 3\cot(2\sqrt{c}r)=\cot(\sqrt{c}r), \nonumber\ee
which has a unique solution  given by $r=\pi/(6\sqrt{c})$ over $(0, \pi/(2\sqrt{c}))$.
If $c<0$, then $M$ is  a hypersurface of type $B$ in $\mathbb{C}H^2(4c)$ satisfying
\be 6\tanh(2\sqrt{|c|}r)=\tanh(\sqrt{|c|}r)+\coth(\sqrt{|c|}r), \nonumber\ee
 which has a unique positive solution given by 
$r=\ln(2+\sqrt{3})/(4\sqrt{|c|})$. 

Conversely, it is easy to verify that these hypersurfaces satisfy identically (\ref{ricci-xi}) by considering
the components of their  shape operators.
\qed

\vskip10pt
\noindent {\it Proof of Theorem} 1.2.
Let $M$ be a Hopf hypersurface  in 
$\tilde M^2(4c)$.
Assume that $M$ admits a unit vector filed $U\in\mathcal{H}$ satisfying the equality in (\ref{ricci-U}) identically. 
It follows from Lemma \ref{lem1} that  there exists an orthonormal frame field
 $\{e_1, e_2, e_3\}$ such that $e_1=\xi$, $\phi e_2=e_3=U$, 
  and the shape operator is given by
 \be A\xi=(3\mu-\gamma)\xi, \quad Ae_2=\gamma e_2,  \quad Ae_3=\mu e_3, \label{3mu}\ee
 for some smooth functions $\gamma$ and $\mu$.
 Theorem \ref{eigen} shows that
 $3\mu-\gamma$ is constant 
  and $2\gamma\mu=(\gamma+\mu)(3\mu-\gamma)+2c$ holds.
 Hence, $M$ is a Hopf hypersurface with constant principal curvatures. 
 By applying Theorems \ref{A1cp}, \ref{Bcp}, \ref{A0}, \ref{A10}, \ref{A11} and \ref{Bch},
 we can determine hypersurfaces satisfying (\ref{3mu}), that is, $\delta=2\lambda$, $3\lambda_1-\lambda_2$ or $3\lambda_2-\lambda_1$ as follows:

If $c>0$, then  $M$ is  a hypersurface of type $B$ in $\mathbb{C}P^2(4c)$ satisfying
 \be 
 2\tan(2\sqrt{|c|}r)=3\tan(\sqrt{|c]}r)+\cot(\sqrt{|c|}r),\nonumber\ee which has a 
  unique solution given by $r=\pi/(6\sqrt{c})$ over $(0, \pi/(4\sqrt{|c|}))$.
  In this case, the principal curvature vector field corresponding to
 $\lambda_2=\sqrt{c/3}$ satisfy (\ref{ricci-U}) identically.
 If $c<0$, then $M$ is 
  a horosphere  in $\mathbb{C}H^2(4c)$. In this case,  the principal curvature vector field corresponding to $\lambda=\sqrt{|c|}$ satisfy (\ref{ricci-U}) identically. \qed

\section{Proof of Theorems 1.3 and 1.4}
{\it Proof of Theorem} 1.3.
Let $M$ be a 2-Hopf hypersurface in  $\tilde M^2(4c)$. 
We choose an orthonormal frame field $\{\xi, X, \phi X\}$ such that the distribution $\mathcal{D}$ spanned by $\{\xi, X\}$ is 
the smallest $A$-invariant distribution.
Then the shape operator $A$ is written by
\begin{align} 
& A\xi=\alpha\xi+\beta X,\ \  AX=\gamma X+\beta\xi, \ \ A\phi X=\mu \phi X \nonumber
 \end{align}
for some functions $\alpha$, $\beta$, $\gamma$ and $\mu$, where $\beta$
is non-zero at each point. Therefore, the shape operator takes the form of (\ref{A}).
According to  Proposition 7 in  \cite{ivey} and its proof, if $\alpha=\<A\xi, \xi\>$ is constant along $\mathcal{D}$-leaves, then 
$\nabla_{\phi X}\phi X=0$ and all the other components of $A$ are also constant 
along $\mathcal{D}$-leaves, and satisfy 
\be
\begin{split}
\frac{d\alpha}{ds}&=\beta(\alpha+\gamma-3\mu),\\
\frac{d\beta}{ds}&=\beta^2+\gamma^2+\mu(\alpha-2\gamma)+c,\\
\frac{d\gamma}{ds}&=\frac{(\gamma-\mu)(\gamma^2-\alpha\gamma-c)}{\beta}
+\beta(2\gamma+\mu),
\end{split} \label{2-hopf}
\ee
where $d/ds$ stands for the derivative with respect to $\phi X$.

If $\<A\xi, \xi\>$ is constant on $M$, then by the first equation in (\ref{2-hopf}) we have $\alpha+\gamma-3\mu=0$, which yields (\ref{B}). Therefore, by applying Lemma \ref{lem1}, 
we see that $\phi X$ satisfies
(\ref{ricci-U}) identically.

 Conversely, suppose that $M$ is a real hypersurface  in  $\tilde M^2(4c)$ which 
 is non-Hopf at every point, and admits a unit geodesic vector field $U\in\mathcal{H}$  satisfying the equality in 
$(\ref{ricci-U})$ identically.
Then it follows from Lemma \ref{lem1} that there exists  an 
   orthonormal frame field $\{e_1, e_2, e_3\}$
   such  that $e_1=\xi$, $\phi e_2=e_3=U$ and the shape operator is given by
\begin{align} 
& A\xi=(3\mu-\gamma)\xi+\beta e_2,\ \  Ae_2=\gamma e_2+\beta\xi, \ \ Ae_3=\mu e_3 \label{A2}
 \end{align}
 for some functions $\beta$, $\gamma$ and $\mu$.
We denote by $\mathcal{V}$  the distribution spanned by $\{\xi, e_2\}$.
 Since $\xi$ is not a principal vector everywhere, 
we have $\beta\ne 0$ on $M$, and therefore,
 $\mathcal{D}$ is invariant under the shape operator and 
 of rank 2.
  We shall prove that $\mathcal{V}$ is integrable and $\<A\xi, \xi\>(=3\mu-\gamma)$ is constant.

By using (\ref{PA}) and (\ref{A2}), we have 
\be
\nabla_{e_2}\xi=\gamma e_3, \ \ \nabla_{e_3}\xi=-\mu e_2, \ \ \nabla_{\xi}\xi=\beta e_3. \label{nab1}
\ee
Since $e_3$ is a unit geodesic vector field, we have $\nabla_{e_3}e_3=0$.
This, together with (\ref{nab1}), yields 
  \be
\begin{split}
&   \nabla_{e_2}e_2 =\chi_1e_3, \ \ \nabla_{e_3}e_2=\mu\xi, \ \ \nabla_{\xi}e_2 =\chi_2e_3, \\
& \nabla_{e_2}e_3 =-\chi_1e_2-\gamma\xi, \ \ \nabla_{\xi}e_3=
-\chi_2e_2-\beta\xi. \label{nab2}
\end{split}
\ee
for some functions  $\chi_1$ and $\chi_2$.
 %


We deduce from (\ref{nab1}), (\ref{nab2}) and    the equation  (\ref{co}) of Codazzi that
\begin{align}
e_2\mu &=0, \label{cd1}\\
e_3\gamma &=(\gamma-\mu)\chi_1+\beta(\gamma+2\mu),\label{cd2}\\
e_3\beta &=-\gamma^2+\beta\chi_1+3\mu^2+2c\label{cd3},\\
e_2\beta &=\xi\gamma,\label{cd4} \\
e_2\gamma &=-\xi\beta, \label{cd5} \\
\beta\chi_1+(\mu-\gamma)\chi_2 &=\beta^2+\gamma^2-2\gamma\mu-c,\label{cd6}\\
\xi\mu &=0, \label{cd7}\\
e_3(3\mu-\gamma) &=\beta(\chi_2-\gamma). \label{cd8} 
\end{align}
By the equation (\ref{ga}) of Gauss for $\<R(e_2, e_3)e_3, e_2\>$ and $\<R(\xi, e_2)e_3, e_2\>$, we obtain
\begin{align}
 & e_3\chi_1-2\mu\gamma-\chi_1^2-(\gamma+\mu)\chi_2-4c=0, \label{ga1}\\
& \xi\chi_1=e_2\chi_2.\label{ga2} 
\end{align}

It follows from (\ref{nab1}), (\ref{nab2}), (\ref{cd1}) and  (\ref{cd7}) that
\be
0=[e_2, \xi]\mu=(\nabla_{e_2}\xi-\nabla_{\xi}e_2)\mu=(\gamma-\chi_2)e_3\mu.\label{eq00}
\ee
Thus, we have  $\gamma=\chi_2$ or $e_3\mu=0$.

{\bf Case A}: $\gamma=\chi_2$. 
In this case, since $\nabla_{e_2}\xi-\nabla_{\xi}e_2=0$ holds, 
  $\mathcal{V}$ is integrable, and therefore,
 $M$ is a 2-Hopf hypersurface.

Equations (\ref{cd6}), (\ref{cd8}) and (\ref{ga1}) are reduced to
\begin{align}
&\beta\chi_1-\beta^2+3\gamma\mu-2\gamma^2+c=0, \label{eq2}\\
& 3e_3\mu=e_3\gamma,\label{eq2.1}\\
&e_3\chi_1=\chi_1^2+\gamma^2+3\gamma\mu+4c, \label{eq2.2}
\end{align}
respectively. 
 Eliminating $\chi_1$ from (\ref{cd3}) and (\ref{eq2}) leads to
\be
e_3\beta=\beta^2+\gamma^2-3\gamma\mu+3\mu^2+c. \label{eq3}
\ee
By combining (\ref{cd5}) and  (\ref{ga2}), we have 
\be
\xi\chi_1=-\xi\beta.\label{eq1}
\ee
By using (\ref{nab1}), (\ref{nab2}), (\ref{cd2}), (\ref{cd4}), (\ref{cd7}), (\ref{eq3}) and (\ref{eq1})  we obtain the following:
\begin{align}
e_3(\xi\beta) &=(\nabla_{e_3}\xi-\nabla_{\xi}e_3)\beta+\xi(e_3\beta) \nonumber \\
&=(\gamma-\mu)\xi\gamma+\beta(\xi\beta)+\xi(\beta^2+\gamma^2-3\gamma\mu+3\mu^2+c),\nonumber\\
&=3\beta(\xi\beta)+(3\gamma-4\mu)\xi\gamma, \label{eq3.1}\\
e_3(\xi\gamma) &=(\nabla_{e_3}\xi-\nabla_{\xi}e_3)\gamma+\xi(e_3\gamma)  \nonumber \\
&=(\mu-\gamma)\xi\beta+\beta(\xi\gamma)+\xi[(\gamma-\mu)\chi_1+\beta(\gamma+2\mu)],\nonumber\\
&=(4\mu-\gamma)\xi\beta+(2\beta+\chi_1)\xi\gamma.\label{eq3.2}
\end{align}

By differentiating (\ref{eq2}) with respect to 
$\xi$, and using (\ref{cd7}) and (\ref{eq1}),  we obtain
\be
(\chi_1-3\beta)\xi\beta+(3\mu-4\gamma)\xi\gamma=0. \label{eq4}
\ee
Moreover, by differentiating (\ref{eq4}) with respect to  $e_3$, and using (\ref{cd2}), (\ref{eq2.1}), (\ref{eq2.2}), (\ref{eq3}), (\ref{eq3.1}) and (\ref{eq3.2}), we have
\be
(\chi_1^2+3\beta\chi_1-12\beta^2+2\gamma^2-7\gamma\mu+3\mu^2+c)\xi\beta+
2[(\mu-2\gamma)\chi_1+(6\mu-10\gamma)\beta]\xi\gamma=0. \label{eq5}
\ee
Equations (\ref{eq4}) and (\ref{eq5}) can be rewritten as 
\be
    \begin{pmatrix}
       a_{11}&a_{12} \\
       a_{21}& a_{22}\\
   \end{pmatrix}
  \begin{pmatrix}
  \xi\beta \\
  \xi\gamma
  \end{pmatrix}=\begin{pmatrix}
  0 \\
  0
  \end{pmatrix}, \label{eq6}
\ee
where the components of the square matrix are given by
\begin{align*}
& a_{11}=\chi_1-3\beta, \\
&a_{12}=3\mu-4\gamma,\\
&a_{21}=\chi_1^2+3\beta\chi_1-12\beta^2+2\gamma^2-7\gamma\mu+3\mu^2+c, \\
&a_{22}=2[(\mu-2\gamma)\chi_1+(6\mu-10\gamma)\beta].
\end{align*}

We divide Case A into two subcases.

{\bf Case A.1}: $a_{11}a_{22}-a_{21}a_{12}=0$.  
Eliminating $\chi_1$ from this equation and (\ref{eq2}) shows
\be
4\beta^4(4\gamma-\mu)+\beta^2(16\gamma^3+2c\mu-56\gamma^2\mu+48\gamma\mu^2-9\mu^3)-\mu(c-2\gamma^2+3\gamma\mu)^2=0. \label{eq7}
\ee 
 Differentiating (\ref{eq7}) with respect to $e_3$ and using (\ref{cd2}),  (\ref{eq2.1}) and 
 (\ref{eq3}), we get
 \be 
\begin{split}
&4\beta^5
   (59\gamma+10\mu)+\beta^3(194c\gamma+376\gamma^3-32c\mu-1024\gamma^2\mu+ 645\gamma\mu^2+ 
       36\mu^3)\\
&\beta\Bigl[92\gamma^5-c^2(\gamma-10\mu)-656\gamma^4\mu+
1617\gamma^3\mu^2- 
   1818\gamma^2\mu^3+918\gamma\mu^4\\
 &-162\mu^5+2c(50\gamma^3-152\gamma^2\mu+129\gamma\mu^2-27\mu^3)\Bigr]\\
       &+(\gamma-\mu)\Bigl[44\beta^4+ 
    \beta^2(2 c+88\gamma^2-240\gamma\mu+117\mu^2)\\
       &+2c(2\gamma^2+6\gamma\mu
       -9\mu^2)-\gamma(4\gamma^3+24\gamma^2\mu-81\gamma\mu^2
    +54\mu^3)-c^2\Bigr]\chi_1=0.
\end{split} \label{eq8}
\ee
Eliminating $\chi_1$ from (\ref{eq8}) and (\ref{eq2}), we have
 \be
 \begin{split}
 &4\beta^6(-70\gamma+\mu)+ 
 \beta^4\Bigl[-552\gamma^3+1572\gamma^2\mu-1134\gamma\mu^2
 +81\mu^3-2c(76\gamma+5\mu)\Bigr]\\
 &+\beta^2\Bigl[c^2(4\gamma-13\mu)
 -c(20\gamma^3+22\gamma^2\mu-123\gamma\mu^2+81\mu^3)\\
 &- 
    3(88\gamma^5-532\gamma^4\mu+1140\gamma^3\mu^2-1086\gamma^2\mu^3+441\gamma\mu^4- 
       54\mu^5)\Bigr]\\
       &-(\gamma-\mu)
 (c-2\gamma^2-15\gamma\mu+18\mu^2)(c-2\gamma^2+3\gamma\mu)^2=0.
 \end{split}\label{eq9}
 \ee
  Eliminating $\beta$ from (\ref{eq9}) and (\ref{eq7}) gives
\be
(4\gamma-3\mu)(c-2\gamma^2+3\gamma\mu)^3f(\gamma, \mu)=0,\label{factor}\nonumber
\ee
where $f(\gamma, \mu)$ is a polynomial given by
\begin{align*}
f(\gamma, \mu)=&4608\gamma^8-28032\mu\gamma^7+(77760\mu^2-64c)\gamma^6
-(133248\mu^3+3168c\mu)\gamma^5\\
&+(155520\mu^4+9696c\mu^2+32c^2)\gamma^4
-(121392\mu^5+21176c\mu^3-528c^2\mu)\gamma^3\\
&+(52920\mu^6+19500c\mu^4+2640c^2\mu^2)\gamma^2
-(7938\mu^7+2556c\mu^5+20c^2\mu^3)\gamma\\
&+243\mu^8+216c\mu^6
+42c^2\mu^4.
\end{align*}

If  $4\gamma-3\mu=0$, then by (\ref{eq2.1})
we get $e_3\mu=0$.
If $c-2\gamma^2+3\gamma\mu=0$, then
differentiating it with respect to  $e_3$ implies $(\mu-\gamma)e_3\mu=0$, and hence,
$(\mu^2+c)e_3\mu=0$, which shows $e_3\mu=0$.
If $f(\gamma, \mu)=0$, then by differentiating $f(\gamma, \mu)=0$ with respect to
 $e_3$
and using (\ref{eq2.1}),  we obtain 
 $g(\gamma, \mu)e_3\mu=0$, 
where 
$g(\gamma, \mu)$ is a non-trivial polynomial in $\gamma$ and $\mu$ which is different from $f(\gamma, \mu)$.
Eliminating $\gamma$ from $f(\gamma, \mu)=0$ and $g(\gamma, \mu)e_3\mu=0$, we get
$p(u)e_3\mu=0$ for a non-trivial polynomial $p(\mu)$ in $\mu$.
We do not list $g(\gamma, \mu)$ and $p(\mu)$ explicitly, however, these polynomials
 can be recovered
quickly by using
a computer algebra program.

Consequently, in any case we have $e_3\mu=0$, which together with
(\ref{cd1}) and (\ref{cd7}) proves that $\mu$ is constant. 
Since $\gamma$ satisfies a polynomial equation with constant coefficients,  $\gamma$  
must be constant. Therefore we conclude that 
$\<A\xi, \xi\>$ is constant.

{\bf Case A.2}: $a_{11}a_{22}-a_{21}a_{12}\ne 0$. 
From (\ref{eq6}) we obtain
\be
\xi\beta=\xi\gamma=0. \label{eq10}\ee
It follows from (\ref{cd5}) and (\ref{eq10}) that $e_2\gamma=0$. 
Using (\ref{cd1}), (\ref{cd7}),  (\ref{eq2.1}) and (\ref{eq10}) yields that $\<A\xi, \xi\>$ is constant.

{\bf Case B}: $e_3\mu=0$. In this case, 
 by (\ref{cd1}) and (\ref{cd7}) we see that $\mu$ is constant.
 
 Combining (\ref{cd2}) and (\ref{cd8}) yields
 \be
 (\gamma-\mu)\chi_1+\beta\chi_2=-2\beta\mu. \label{eq11}
 \ee
 Solving (\ref{cd6}) and (\ref{eq11}) for $\chi_1$ and $\chi_2$, we get
\be
\begin{split}
\chi_1&=\frac{\beta(\beta^2+\gamma^2-4\gamma\mu+2\mu^2-c)}
{\beta^2+(\gamma-\mu)^2},\\
\chi_2&=\frac{(c-\gamma^2+2\gamma\mu)(\gamma-\mu)-\beta^2(\gamma+\mu)}
{\beta^2+(\gamma-\mu)^2}.\label{eq12}
\end{split}
\ee 
 Substituting  (\ref{eq12}) into (\ref{ga1}) and using (\ref{cd2}) and (\ref{cd3}), we obtain
 \be
 (\gamma-\mu)h(\beta, \gamma)=0,
 \ee
where $h(\beta, \gamma)$ is given by  the following function:
\be 
 h(\beta, \gamma)=
 (c-2\mu^2)(\beta^2+\gamma^2)+6\mu^3\gamma-3\mu^4+c^2.\nonumber\ee

We divide  Case B into two subcases.

{\bf Case B.1}: $\gamma-\mu=0$. In this case, $\gamma$ is constant. By (\ref{cd2}), we get
$\gamma+2\mu=0$, which shows that $\gamma=\mu=0$. 
From (\ref{cd8}) we have $\chi_2=0$. Consequently, $M$  is a minimal
2-Hopf hypersurface with
$\<A\xi, \xi\>=0$.

{\bf Case B.2}: $h(\beta, \gamma)=0$.  We find that 
the differentiation of  this equation with respect to  $e_3$ gives us no information. 
Thus, we differentiate  $h(\beta, \gamma)=0$
with respect to  $\xi$ and $e_2$. Then, using (\ref{cd4}) and (\ref{cd5}), we get
\be
    \begin{pmatrix}
       h_{\beta}&h_{\gamma} \\
       -h_{\gamma}& h_{\beta}\\
   \end{pmatrix}
  \begin{pmatrix}
  \xi\beta \\
  \xi\gamma
  \end{pmatrix}=\begin{pmatrix}
  0 \\
  0
  \end{pmatrix},\label{h}
\ee
where $h_{\beta}$ and $h_{\gamma}$ denote partial derivatives of $h$ with respect to $\beta$ and $\gamma$, respectively. 
If $h_{\beta}^2+h_{\gamma}^2=0$, then we obtain $c=0$. Hence this case cannot occur.
Thus, we deduce from (\ref{h}), ({\ref{cd4}) and (\ref{cd5}) that
\be\xi\beta=\xi\gamma=e_2\beta=e_2\gamma=0. \label{eq13}\ee
By (\ref{nab1}), (\ref{nab2}) and (\ref{eq13}), we have
\be
0=e_2(\xi\gamma)-\xi(e_2\gamma)=(\nabla_{e_2}\xi-\nabla_{\xi}e_2)\gamma=(\gamma-\chi_2)e_3\gamma.\nonumber
\ee
Combining this and (\ref{cd8}), we obtain 
$\gamma-\chi_2=e_3\gamma=0$, which together with (\ref{eq13}) yields that $\mathcal{V}$ is integrable and 
$\gamma$ is constant. As a consequence, $M$ is a 2-Hopf hypersurfaces such that
$\<A\xi, \xi\>$  is   constant. \qed



\begin{remark}\label{rem1}
{\rm 
We solve  the system (\ref{2-hopf}) of ODE's under the condition that
$\alpha$, $\beta$ and $\gamma$  are constant. Then $c<0$ and the shape operator can be  expressed as
\be
    A=\sqrt{-c}\begin{pmatrix}
       3u-u^3& (1-u^2)^{\frac{3}{2}}& 0 \\
        (1-u^2)^{\frac{3}{2}}& u^3 & 0\\
        0 & 0 & u
   \end{pmatrix}\nonumber
\ee
with respect to an orthonormal frame field $\{\xi, X, \phi X\}$,
where $u$ is  a  constant in the range $-1<u<1$. 
If $u=0$, then $M$ is  a ruled minimal homogeneous 
hypersurface  $W^3$ in $\mathbb{C}H^2(4c)$ which was introduced by
Lohnherr (see \cite{loh}), otherwise,  $M$ is one of the equidistant hypersurfaces to $W^3$
(see \cite{bera} and Section 6.4 in \cite{ivey2}). 

}
 \end{remark}
\begin{remark}\label{rem2}
{\rm 
It follows from  \cite[Proposition 8.27]{cecil} that
 a real hypersurface $M$ in $\tilde M^2(4c)$ 
 is a  minimal ruled hypersurface if and only if there exists a unit vector field $X$ on $M$, which is 
 orthogonal to $\xi$ and satisfies 
 \be 
A\xi=\beta X, \ \ AX=\beta\xi, \ \ A\phi X=0. \label{ruled}\nonumber
\ee
Thus, a hypersurface described  in Case B.1 in the proof  of Theorem \ref{thm3}
is nothing but a minimal ruled hypersurface.
Minimal ruled real hypersurfaces in non-flat complex space forms  have been classified in \cite{ada}. }
\end{remark}

\medskip

{\it Proof of Theorem }\ref{thm4}.
The hypersurfaces described in Theorem \ref{thm4} have constant mean curvature.
Since  their shape operators are expressed as
(\ref{A}) and satisfy (\ref{B}), by  the ``if" part of Lemma \ref{lem1} we see that
these hypersurfaces 
admit 
 a unit vector field $U\in\mathcal{H}$  satisfying the equality in 
$(\ref{ricci-U})$ identically.
 
 Conversely, suppose that $M$ be a real hypersurface with constant mean curvature
   in  $\tilde M^2(4c)$ which 
 is non-Hopf at every point, and admits a unit vector field $U\in\mathcal{H}$  satisfying the equality in 
$(\ref{ricci-U})$ identically.
Similarly to the proof of Theorem \ref{thm3}, we can choose an 
   orthonormal frame field $\{e_1, e_2, e_3\}$
   such that that $e_1=\xi$, $\phi e_2=e_3=U$ and
    the shape operator takes the form (\ref{A2}). Then (\ref{nab1}) holds.
However, since we do not assume that $U$ is a geodesic vector field, we have
$\nabla_{e_3}e_3=\chi_3e_2$ for some function $\chi_3$. Therefore, the equation
$\nabla_{e_3}e_2=\mu\xi$ in (\ref{nab2}) is replaced by $\nabla_{e_3}e_2=-\chi_3e_3+\mu\xi$.

From the equation (\ref{co}) of Codazzi for $X=e_3$ and $Y=\xi$, comparing
the coefficient of $e_3$,  we obtain $\xi\mu=-\beta\chi_3$ instead of (\ref{cd7}).  By (\ref{A2}), the constancy of the 
mean curvature yields that $\mu$ is constant. Hence, we get $\chi_3=0$, that is, $e_3$ is a geodesic vector field. By Theorem \ref{thm3}, $M$ is a 2-Hopf hypersurface such that $\<A\xi, \xi\>(=3\mu-\gamma)$ is constant.
Since $\mu$ is constant,  $\gamma$ is also constant.
Hence, the third equation in (\ref{2-hopf}) can be reduced to
\be
(\gamma-\mu)(2\gamma^2-3\gamma\mu-c)+\beta^2(2\gamma+\mu)=0.\label{loh}
\ee  

If $2\gamma+\mu=0$, then $\gamma=\mu=0$ or $\mu=-2\gamma=\pm\sqrt{c/2}\ \ (c>0)$. 
In the former case, it follows from  Remark \ref{rem2} 
that $M$ is a minimal ruled real hypersurface.
In the latter case, 
by the second equation in (\ref{2-hopf}) with $\alpha=3\mu-\gamma$,
we obtain $\beta(s)=\sqrt{27c/8}\tan(\sqrt{27c/8}s+d)$ for some constant $d$. Therefore, $M$ is a
hypersurface described in Case (ii) of Theorem \ref{thm4}.
If $2\gamma+\mu\ne 0$, then $\beta$ must be constant, and therefore, it follows from Remark \ref{rem1}
that $M$ is one of the equidistant hypersurfaces to Lohnherr hypersurface  in $\mathbb{C}H^2(4c)$.\qed


 \end{document}